\pdfoutput=1
\RequirePackage{ifpdf}
\ifpdf 
\documentclass[pdftex]{sigma}
\else
\documentclass{sigma}
\fi

\usepackage{diagrams}

\makeatletter
\def\foo#1\endgraf\unskip#2\foo{\def\row@to@buffer{#1\endgraf\unskip\unskip#2}}
\expandafter\foo\row@to@buffer\foo
\makeatother

\numberwithin{equation}{section}

\newtheorem{TheoremABC}{Theorem}

\newtheorem{CorollaryABC}[TheoremABC]{Corollary}
\newtheorem{ConjectureABC}[TheoremABC]{Conjecture}

\newtheorem{Theorem}{Theorem}[subsection]
\newtheorem*{Theorem*}{Theorem}
\newtheorem{Corollary}[Theorem]{Corollary}
\newtheorem{Lemma}[Theorem]{Lemma}
\newtheorem{Proposition}[Theorem]{Proposition}
 { \theoremstyle{definition}
\newtheorem{Definition}[Theorem]{Definition}

\newtheorem{Remark}[Theorem]{Remark} }

\renewcommand{\k}{\mathbf{k}}

\newcommand{\id}{\operatorname{id}}

\newcommand{\we}{\wedge}

\newcommand{\rk}{\operatorname{rk}}

\newcommand{\WW}{{\mathcal W}}
\newcommand{\EE}{{\mathcal E}}

\newcommand{\VV}{{\mathcal V}}

\newcommand{\LL}{{\mathcal L}}

\newcommand{\Hom}{\operatorname{Hom}}

\newcommand{\Ext}{\operatorname{Ext}}
\newcommand{\End}{\operatorname{End}}

\renewcommand{\a}{\alpha}
\renewcommand{\b}{\beta}
\newcommand{\om}{\omega}

\newcommand{\la}{\lambda}

\renewcommand{\P}{{\mathbb P}}

\newcommand{\La}{\Lambda}

\newcommand{\wt}{\widetilde}
\newcommand{\ot}{\otimes}

\newcommand{\sub}{\subset}

\newcommand{\GL}{\operatorname{GL}}

\renewcommand{\k}{\mathbf{k}}

\renewcommand{\mod}{\operatorname{mod}}

\newcommand{\und}{\underline}

\newcommand{\OO}{{\mathcal O}}

\newcommand{\Si}{\Sigma}

\newcommand{\DD}{{\mathcal D}}

\newcommand{\hra}{\hookrightarrow}
\newcommand{\lan}{\langle}
\newcommand{\ran}{\rangle}
\newcommand{\Coh}{\operatorname{Coh}}

\newcommand{\CC}{{\mathcal C}}
\newcommand{\UU}{{\mathcal U}}

\renewcommand{\P}{{\mathbb P}}

\newcommand{\Pic}{\operatorname{Pic}}

\newcommand{\de}{\delta}

\renewcommand{\ker}{\operatorname{ker}}
\newcommand{\im}{\operatorname{im}}

\newcommand{\pa}{\partial}
\newcommand{\fg}{{\mathfrak g}}

\newcommand{\Gr}{G}

\begin{document}

\allowdisplaybreaks

\newcommand{\arXivNumber}{2310.18759}

\renewcommand{\PaperNumber}{037}

\FirstPageHeading

\ShortArticleName{Compatible Poisson Brackets Associated with Elliptic Curves in $G(2,5)$}

\ArticleName{Compatible Poisson Brackets Associated\\ with Elliptic Curves in $\boldsymbol{G(2,5)}$}

\Author{Nikita MARKARIAN~$^{\rm a}$ and Alexander POLISHCHUK~$^{\rm bc}$}

\AuthorNameForHeading{N.~Markarian and A.~Polishchuk}

\Address{$^{\rm a)}$~Universit\'e de Strasbourg, France}
\EmailD{\href{mailto:nikita.markarian@gmail.com}{nikita.markarian@gmail.com}}

\Address{$^{\rm b)}$~Department of Mathematics, University of Oregon, Eugene, OR~97403, USA}
\EmailD{\href{mailto:apolish@uoregon.edu}{apolish@uoregon.edu}}

\Address{$^{\rm c)}$~National Research University Higher School of Economics, Moscow, Russia}

\ArticleDates{Received December 05, 2023, in final form April 27, 2024; Published online May 07, 2024}

\Abstract{We prove that a pair of Feigin--Odesskii Poisson brackets on ${\mathbb P}^4$ associated with elliptic curves given as linear sections of the Grassmannian $G(2,5)$ are compatible if and only if this pair of elliptic curves is contained in a del Pezzo surface obtained as a linear section of $G(2,5)$.}


\Classification{14H52; 53D17}

\section{Introduction}

We work over an algebraically closed field $\k$ of characteristic $0$.

In this paper we continue to study compatible pairs among the Poisson brackets on projective spaces introduced by Feigin--Odesskii (see \cite{FO95,P98}).
Their construction associates with every stable vector bundle $\VV$ of degree $n>0$ and rank $k$ on an elliptic curve $E$, a Poisson bracket on the projective
space $\P H^0(E,\VV)^*$. We refer to such Poisson brackets as FO brackets of type $q_{n,k}$.

Two Poisson brackets are called {\it compatible} if the corresponding bivectors satisfy ${[\Pi_1,\Pi_2]=0}$ (equivalently, any linear combination of these brackets is again Poisson).
In \cite{OW}, Odesskii and Wolf discovered $9$-dimensional spaces of compatible FO brackets of type $q_{n,1}$ on $\P^{n-1}$
for each~${n\ge 3}$. Their construction was interpreted and extended in \cite{HP-bih}, where the authors showed that one gets compatible FO brackets
if the elliptic curves are anticanonical divisors on a surface~$S$ and the stable bundles on them are restrictions of a single
exceptional bundle on $S$ that forms an exceptional pair with $\OO_S$ (see \cite[Theorem~4.4]{HP-bih}).
One can ask whether any two compatible FO brackets of type $q_{n,k}$ on $\P^{n-1}$ appear in this way.
In \cite{MP} we have shown that this is the case for $k=1$ (for some specific rational surfaces containing normal elliptic curves in projective spaces).
In the present work, we consider the case of FO brackets of type $q_{5,2}$ on $\P^4$.
Note that the question of finding bi-Hamiltonian structures with brackets of type $q_{5,2}$ was raised by Rubtsov in~\cite{Rubtsov}.

Let $V$ be a $5$-dimensional vector space.
Consider the Pl\"ucker embedding
\begin{gather*}
G(2,V)\to \P\left({\bigwedge}^2V\right).
\end{gather*}
It is well known that for a generic $5$-dimensional subspace $W\sub {\bigwedge}^2 V$ the corresponding linear section
\begin{gather*}E_W:=G(2,V)\cap \P W\end{gather*}
is an elliptic curve. Furthermore, if $\UU\sub V\ot \OO$ is the universal subbundle on $G(2,V)$, then one can check that the restriction $V_W:=\UU^\vee|_{E_W}$ is a
stable bundle of rank $2$ and degree $5$ on $E_W$
(see Lemma \ref{Gr-res-map-lem} below).
Thus, we have the corresponding Feigin--Odesskii bracket of type $q_{5,2}$ on $\P H^0(E_W,V_W)^*$.

Furthermore, one can check that the restriction map
\begin{gather*}V^*=H^0(G(2,V),\UU^\vee)\to H^0(E_W,V_W)\end{gather*}
is an isomorphism (see Lemma \ref{Gr-res-map-lem}). Thus, we get a Poisson bracket $\Pi_W$ on $\P V$ (defined up to a rescaling).

On the other hand, we have a natural $\GL(V)$-invariant map
\begin{gather*}\pi_{5,2}\colon\ {\bigwedge}^5\left({\bigwedge}^2V\right)\to H^0\left(\P V,{\bigwedge}^2 T\right)\ot {\det}^2(V)\end{gather*}
constructed as follows.

Note that we have a natural isomorphism $V\simeq H^0(\P V, T(-1))$, hence we get a natural map~${V\ot \OO(1)\to T}$, and hence, the composed map
\begin{gather*}\phi\colon\ W\ot \OO(2)\to {\bigwedge}^2 V\ot \OO(2)\to  {\bigwedge}^2 T\end{gather*}
on $\P V$. Taking the 5th exterior power of this map, we get a map
\begin{gather*}{\bigwedge}^5(\phi)\colon\ \det(W)\ot \OO(10)\to {\bigwedge}^5\left({\bigwedge}^2 T\right)\simeq \left({\bigwedge}^2 T\right)^\vee \ot {\det}^3(T),\end{gather*}
where we used the identification $\det\big({\bigwedge}^2 T\big)\simeq \det^3(T)$.
Note that we have a nondegenerate pairing given by the exterior product,
\begin{gather*}{\bigwedge}^2 T\ot {\bigwedge}^2 T\to \det(T),\end{gather*}
hence, we have an isomorphism ${\bigwedge}^2 T\simeq \big({\bigwedge}^2 T\big)^\vee \ot \det(T)$, and we can rewrite the above map as
\begin{gather*}\det(W)\to {\bigwedge}^2 T\ot {\det}^2(T)(-10)\simeq {\bigwedge}^2 T\ot {\det}^2(V).\end{gather*}

\begin{TheoremABC}\label{TheoremA}
For every $5$-dimensional subspace $W\sub {\bigwedge}^2 V$, such that $E_W:=G(2,V)\cap \P W$ is an elliptic curve, one has
an equality
\[
\pi_{5,2}(\la_W)=\Pi_W\ot \de,
\]
for some
trivializations $\la_W\in {\bigwedge}^5 W$ and $\de\in \det^2(V)$.
\end{TheoremABC}

Theorem~\ref{TheoremA} is deduced from the existence of a formula for $\Pi_W$, depending linearly on the Pl\"ucker coordinates of $W$ (which follows from the results of \cite{HP-bih}),
combined with a representation-theoretic argument employing the fact that the construction of $\Pi_W$ is $\GL(V)$-equivariant.

\begin{TheoremABC}\label{TheoremB} \quad
\begin{enumerate}\itemsep=0pt
\item[$(i)$] For $5$-dimensional subspaces $W,W'\sub {\bigwedge}^2 V$ such that $E_W$ and $E_{W'}$ are elliptic curves, the Poisson brackets $\Pi_W$ and $\Pi_{W'}$
are compatible if and only if $\dim W\cap W'\ge 4$.

\item[$(ii)$] For any collection $(W_i)$  of  $5$-dimensional subspaces in ${\bigwedge}^2 V$,
the brackets $(\Pi_{W_i})$ are pairwise compatible if and only if either there exists a $6$-dimensional
subspace $U\sub {\bigwedge}^2 V$ such that each $W_i$ is contained in $U$, or
there exists a $4$-dimensional subspace $K\sub {\bigwedge}^2 V$ such that each $W_i$ contains $K$.
\end{enumerate}
\end{TheoremABC}

The idea of proof is to analyze the vanishing $[\Pi_{W_1},\Pi_{W_2}]=0$ near a sufficiently generic point where $\Pi_{W_1}$ vanishes.
An important ingredient of the proof is a $2$-dimensional distribution on~$G(2,V)$ associated with $W\sub {\bigwedge}^2 V$:
it corresponds to the rational map from $G(2,V)$ to $\P^4$ obtained as the composition of the Plucker embedding with the linear projection to $\P({\bigwedge}^2 V/W)$
(see Section~\ref{Td-distr-sec}). The analysis of the vanishing of the Schouten bracket is used to prove that the elliptic curve $E_{W_1}$ is everywhere tangent to the distribution
associated with $W_2$, which implies the result.

\begin{CorollaryABC}\label{CorollaryC}
The maximal dimension of a linear subspace of Poisson brackets on $\P(V)$, where $\dim V=5$, spanned by some FO brackets $\Pi_W$ of type $q_{5,2}$,
is $6$.
\end{CorollaryABC}

Theorems \ref{TheoremA} and~\ref{TheoremB} suggest the following

\begin{ConjectureABC}\label{ConjectureD} Let $W\sub {\bigwedge}^2 V$ be a $5$-dimensional subspace such that $E_W$ is an elliptic curve.
Consider the subspace
\begin{gather*}T_W:=\left({\bigwedge}^4 W\right)\we \left({\bigwedge}^2V\right)\sub {\bigwedge}^5\left({\bigwedge}^2 V\right)\end{gather*}
$($the quotient of the latter subspace by ${\bigwedge}^5 W$ is exactly the image of the tangent
space to the Grassmannian $G\big(5,{\bigwedge}^2V\big)$ under Pl\"ucker embedding$)$.
Then the subspace of $\xi\in {\bigwedge}^5\big({\bigwedge}^2V\big)$ satisfying~${[\pi_{5,2}(\xi),\Pi_W]=0}$ coincides with
$T_W+\ker(\pi_{5,2})$.
\end{ConjectureABC}

Note that we know the inclusion one way: the subspace $T_W$ is spanned by ${\bigwedge}^5(W')$
such that $\dim(W'\cap W)\!\ge \!4$ and $E_{W'}$ is an elliptic curve, and by Theorems~\ref{TheoremA} and~\ref{TheoremB}, ${\big[\pi_{5,2}\big({\bigwedge}^5\!(W')\big),\Pi_W\big]\!=\!0}$.

\section{Generalities}

\subsection[Feigin--Odesskii Poisson brackets of type q\_\{n,k\}]{Feigin--Odesskii Poisson brackets of type $\boldsymbol{q_{n,k}}$}

Let $E$ be an elliptic curve, with a fixed trivialization $\eta\colon\OO_E\to \om_E$, $\VV$ a stable bundle on $E$ of rank~$k$ and degree $n>0$.
We consider the corresponding Feigin--Odesskii Poisson bracket ${\Pi=\Pi_{E,\VV}}$ of type $q_{n,k}$ on the projective space $\P H^1\big(E,\VV^\vee\big)$ defined as in \cite{P98}.

We will need the following definition of $\Pi$ in terms of triple Massey products. For nonzero $\phi\in H^1\big(E,\VV^\vee\big)$, we denote by $\lan \phi\ran$ the corresponding line,
and we use the identification of the cotangent space to $\lan \phi\ran$ with $\lan \phi\ran^\perp\sub H^0(E,\VV)$ \big(where we use the Serre duality $H^0(E,\VV)\simeq H^1\big(E,\VV^\vee\big)^*$\big).

\begin{Lemma}[{\cite[Lemma 2.1]{HP-bih}}]\label{FO-gen-formula-lem}
For $s_1,s_2\in \lan\phi\ran^\perp$ one has
\begin{gather*}\Pi_{\phi}(s_1\we s_2)=\lan\phi,MP(s_1,\phi,s_2)\ran,\end{gather*}
where $MP$ denotes the triple Massey product for the arrows
\begin{gather*}\OO\rTo{s_2} \VV\rTo{\phi} \OO[1]\rTo{s_1} \VV[1].\end{gather*}
\end{Lemma}

\subsection{Formula for a family of complete intersections}\label{Koszul-sec}

Let $X$ be a smooth projective variety of dimension $n$, $C\sub X$ a connected curve given as the zero locus of a regular section $F$ of a vector bundle $N$ of rank $n-1$,
such that $\det(N)^{-1}\simeq \om_X$. Then the normal bundle to $C$ is isomorphic to $N|_C$, so by the adjunction formula, $\om_C$ is trivial.
Thus, if $C$ is smooth, it is an elliptic curve.
Assume that $P$ is a vector bundle on $X$, such that the following cohomology vanishing holds:
\begin{equation}\label{N-coh-van-eq}
H^i\left(X,{\bigwedge}^iN^\vee\ot P\right)=H^{i-1}\left(X,{\bigwedge}^iN^\vee\ot P\right)=0 \qquad \text{for}\quad 1\le i\le n-1.
\end{equation}

We have the following Koszul resolution for $\OO_C$:
\begin{gather*}0\to {\bigwedge}^{n-1}N^\vee\to \cdots\to {\bigwedge}^2N^\vee\rTo{\de_2(F)} N^\vee \rTo{\de_1(F)} \OO_X\to \OO_C\to 0,\end{gather*}
which induces a map $e_C\colon \OO_C\to {\bigwedge}^{n-1}N^\vee[n-1]$ in the derived category of $X$.
Here the differential~$\de_i(F)$ is given by the contraction with $F\in H^0(X,N)$, so it depends linearly on $F$.

\begin{Lemma}\label{Gr-res-map-lem}\quad
\begin{itemize}\itemsep=0pt
\item[$(i)$] The natural restriction map $H^0(X,P)\to H^0(C,P|_C)$ and the map
\begin{gather*}\Ext^1(P,\OO_C)\rTo{e_C} \Ext^n\left(P,{\bigwedge}^{n-1}N^\vee\right)\simeq \Ext^n(P,\om_X)\end{gather*}
are isomorphisms. These maps are dual via the Serre duality isomorphisms
\begin{gather*}\Ext^1(P|_C,\OO_C)\simeq H^0(C,P|_C)^*, \qquad \Ext^n(P,\om_X)\simeq H^0(X,P)^*.\end{gather*}
\item[$(ii)$] Assume in addition that $\End(P)=\k$ and we have the following vanishing:
\begin{equation}\label{F-N-coh-van-eq}
\Ext^i\left(P,{\bigwedge}^iN^\vee\ot P\right)=\Ext^{i-1}\left(P,{\bigwedge}^iN^\vee\ot P\right)=0 \qquad \text{for}\quad 1\le i\le n-1.
\end{equation}
Then the bundle $P|_C$ is stable.
\end{itemize}
\end{Lemma}

\begin{proof}
(i) This is obtained from the Koszul resolution of $\OO_C$. For example, the space $H^0(P\ot \OO_C)$ is computed by tensoring this resolution
with $P$ and using the spectral sequence
\[H^i\left({\bigwedge}^j N^\vee\right) \Rightarrow H^{i-j}(P\ot \OO_C)\]
 and the assumption \eqref{N-coh-van-eq}.

(ii) Computing $\Hom(P|_C,P|_C)=\Hom(P,P|_C)$ using the Koszul resolution of $P|_C=P\ot \OO_C$,
we get that it is $1$-dimensional. Hence, $P|_C$ is stable.
\end{proof}

Now we can rewrite the formula of Lemma \ref{FO-gen-formula-lem} for the FO-bracket $\Pi_{C,P|_C}$ on
\[\P H^1\big(C,P^\vee|_C\big)\simeq \P \Ext^n(P,\om_X)\]
in terms of higher products on $X$ (obtained by the homological perturbation from a dg-en\-hance\-ment of $D^b(\Coh(X))$).

\begin{Proposition}\label{complete-int-prop}
For nonzero $\phi\in \Ext^n(P,\om_X)\simeq \Ext^1_C(P|_C,\OO_C)$, and $s_1,s_2\in \lan\phi\ran^\perp\sub H^0(X,P)$, one has
\begin{gather*}\Pi_{C,P|_C,\phi}(s_1\we s_2)=\pm\!\left\langle \!\!\phi, \!\sum_{i=1}^n (-1)^i m_{n+2}(\de_1(F),\ldots,\de_{i-1}(F),s_1,\de_i(F),\ldots,\de_{n-1}(F),\phi,s_2)\!\right\rangle.\end{gather*}
\end{Proposition}

\begin{proof} The computation is completely analogous to that of \cite[Proposition 3.1]{NP}, so we will only sketch it.
First, one shows that our Massey product can be computed as the triple product $m_3$ for the arrows
\begin{gather*}\OO_X\to P\rTo{[1]}\OO_C\to P|_C\end{gather*}
given by $s_2$, $\phi$ and $s_1$.
Then we use resolutions ${\bigwedge}^\bullet N^\vee\to \OO_C$ and ${\bigwedge}^\bullet N^\vee\ot P\to P|_C$.
Thus, we have to calculate the following triple product in the category of twisted complexes:
\[
\begin{diagram}
\OO_X\\
\dTo{s_2}\\
P\\
\dTo{\phi}\\
{\bigwedge}^{n-1}N^\vee [n-1]&\rTo{\de_{n-1}(F)}&\cdots&\rTo{\de_2(F)}&N^\vee[1]&\rTo{\de_1(F)}& \OO_X\\
\dTo{s_1}&&&&\dTo{s_1}&&\dTo{s_1}\\
{\bigwedge}^{n-1}N^\vee\ot P [n-1]&\rTo{\de_{n-1}(F)}&\cdots&\rTo{\de_2(F)}&N^\vee\ot P[1]&\rTo{\de_1(F)}& P,\\
\end{diagram}
\]
where we view $\phi$ as a morphism of degree $1$ from $P$ to the twisted complex $\bigoplus {\bigwedge}^iN^\vee [i]$.
Now, the result follows from the formula for $m_3$ on twisted complexes (see \cite[Section 7.6]{Keller}).
\end{proof}

\subsection{Conormal Lie algebra}

Let $\VV$ be a stable bundle of positive degree on an elliptic curve $E$, with a fixed trivialization of $\om_E$, and consider the corresponding FO bracket $\Pi$ on
the projective space $X=\P H^0(\VV)^*=\P \Ext^1(\VV,\OO)$. Recall that for every point $x$ of a smooth Poisson variety $(X,\Pi)$ there is a natural Lie algebra structure on
\begin{gather*}\fg_x:=(\im \Pi_x)^\perp\sub T_x^*X,\end{gather*}
where we consider $\Pi_x$ as a map $T^*_xX\to T_xX$.
We call $\fg_x$ the {\it conormal Lie algebra}. In the case when $\Pi$ vanishes on $x$, we have $\fg_x=T_x^*$.

Let us consider a nontrivial extension
\begin{gather*}0\to \OO\rTo{i}\wt{\VV} \rTo{p} \VV\to 0\end{gather*}
with the class $\phi\in \Ext^1(\VV,\OO)$. By Serre duality, we have the corresponding hyperplane ${\lan\phi\ran^\perp\sub H^0(\VV)}$, and we have an identification
$\lan\phi\ran^\perp\simeq T_\phi^*\P H^0(\VV)^*$.

Consider a natural map
\begin{equation}\label{conormal-map}
\End\big(\wt{\VV}\big)/\lan\id\ran\to \lan\phi\ran^\perp\simeq T_\phi^*\P H^0(\VV)^*\colon\ A\mapsto p\circ A\circ i.
\end{equation}

The following result was proved in \cite{Gor}.

\begin{Theorem}\label{conormal-thm}
The above map induces an isomorphism of Lie algebras from $\End\big(\wt{\VV}\big)/\lan \id\ran$ to the conormal Lie algebra of $\Pi$ at the point $\phi$.
\end{Theorem}

Note that in particular, the subspace $(\im \Pi_x)^\perp\sub \lan\phi\ran^\perp$ is equal to the image of the map \eqref{conormal-map}.

\section[FO brackets associated with elliptic curves in G(2,5)]{FO brackets associated with elliptic curves in $\boldsymbol{G(2,5)}$}

\subsection[Proof of Theorem A]{Proof of Theorem \ref{TheoremA}}

\begin{Lemma}\label{codim2-lem}
The subset $Z\sub \Gr\big(5,{\bigwedge}^2 V\big)$ of $5$-dimensional subspaces $W\sub {\bigwedge}^2 V$ such that $\dim(\P W\cap G(2,V))\ge 2$ has codimension $>1$.
\end{Lemma}

\begin{proof}
Let us denote by $F$ the variety of flags $L\sub W\sub {\bigwedge}^2 V$, where $\dim(L)=3$, $\dim(W)=5$, such that $\P L\cap G(2,V)\neq\varnothing$.
We claim that $F$ is irreducible of dimension $\le 30$. Note that we have a proper closed subset $\wt{Z}\sub F$ consisting of $(L,W)$ such that $\dim(\P W\cap G(2,V))\ge 2$
(as an example of a point in $F\setminus \wt{Z}$, we can take $W$ such that $E_W=\P W\cap G(2,V)$ is an elliptic curve and pick $\P L\sub \P W $ intersecting $E_W$).
Since $\wt{Z}$ fibers over $Z$ with fibers $\Gr(3,5)$, our claim would imply that $\dim\big(\wt{Z}\big)=\dim Z+6<30$, i.e., $\dim Z<24$, as required.

To estimate the dimension of $F$, we observe that we have a fibration $F\!\to Y$ with fibers~$G(2,7)$,
where $Y\sub \Gr\big(3,{\bigwedge}^2 V\big)$ is the subvariety of $3$-dimensional subspaces $L$ such that
$\P L\cap G(2,V)\neq \varnothing$. Thus, it is enough to prove that $Y$ is irreducible of dimension $\le 20$. Now we use a surjective map $\wt{Y}\to Y$, where
$\wt{Y}$ is the variety of flags $\ell\sub L\sub {\bigwedge}^2 V$, where $\dim(\ell)=1$, $\dim(L)=3$, such that $\ell\in G(2,V)$. We have a fibration
$\wt{Y}\to G(2,V)$ with fibers $G(2,9)$, hence~$\wt{Y}$ is irreducible of dimension $6+14=20$. Hence, $Y$ is irreducible of dimension $\le 20$.
\end{proof}

\begin{proof}[Proof of Theorem \ref{TheoremA}]
First, we can apply Proposition \ref{complete-int-prop} to an elliptic curve $E_W\sub X=G(2,V)$. Namely,
as a bundle $P$ on $X$ we take $\UU^\vee$, the dual of the universal subbundle.
We can view the embedding
\begin{gather*}R:=W^\perp\to {\bigwedge}^2V^*=H^0(X,\OO(1)),\end{gather*}
where $\OO(1)=\det(\UU^\vee)$, as a regular section
$F\in H^0(X,N)$, where $N=R^*\ot \OO(1)$.
It is easy to see that we have a $\GL(V)$-invariant identification
\begin{gather*}\om_X\simeq \det(V)^{-2}\ot \OO(-5).\end{gather*}
Thus, by adjunction we get an isomorphism
\begin{gather*}\om_{E_W}\simeq \det(N)\ot \om_X|_{E_W}\simeq \det(R^*)\ot \det(V)^{-2}\ot \OO_{E_W}.\end{gather*}
Since $\det(R^*)\simeq \det\big({\bigwedge}^2 V\big)\ot \det(W^*)\simeq \det(V)^4\ot \det(W^*)$, we can rewrite this as
\begin{equation}\label{om-EW-eq}
\om_{E_W}\simeq \det(W^*)\ot \det(V)^2\ot \OO_{E_W}.
\end{equation}

The vanishings \eqref{N-coh-van-eq} and \eqref{F-N-coh-van-eq} in this case follow from the
well known vanishings
\begin{gather*}H^*\big(X,\UU^\vee(-i)\big)=0 \qquad \text{for}\quad 1\le i\le 5,\\
\Ext^*\big(\UU^\vee,\UU^\vee(-i)\big)=0 \qquad\text{for}\quad 1\le i\le 3, \\
\Ext^{<6}\big(\UU^\vee,\UU^\vee(-4)\big)=\Ext^{<6}\big(\UU^\vee,\UU^\vee(-5)\big)=0
\end{gather*}
(see~\cite{Kap}).
Thus, Proposition \ref{complete-int-prop} gives a formula for $\Pi_W$.

This shows that the association $W\mapsto \Pi_W$ gives a regular morphism
\begin{gather*}f\colon\ \Gr\left(5,{\bigwedge}^2V\right)\to \P H^0\left(\P V,{\bigwedge}^2T\right).\end{gather*}

Furthermore, we claim that
\begin{gather*}f^*\OO(1)\simeq \OO_{\Gr(5,{\bigwedge}^2V)}(1)\ot \det(V)^{2}.\end{gather*}
Indeed, we have a family of Gorenstein curves $\pi\colon \CC\to B=\Gr\big(5,{\bigwedge}^2 V\big)\setminus Z$ (with $\CC_W=E_W$),
where $Z$ was defined in Lemma \ref{codim2-lem}, such that
\begin{gather*}\om_{\CC/B}\simeq \pi^*\big(\OO(1)\ot \det(V)^2\big).\end{gather*}
Indeed, this is implied by the argument leading to \eqref{om-EW-eq}, which works for any curve (not necessarily smooth) cut out by $\P W$ in $G(2,V)$.
This family of curves is equipped with a family of vector bundles $\VV$ (the pull-back of $\UU^\vee$ on $G(2,V)$), so that $\P H^0(\CC_W,\VV_W)^\vee=\P V$.
As explained in \cite[Section~4.2]{HP-bih},
we can view the corresponding constant family of projective spaces $\P V\times B$ as the coarse moduli space of a substack in the relative moduli of complexes on $\CC$.
Now \cite[Proposition~4.1]{HP-bih} implies that the relation $f^*\OO(1)=\OO(1)\ot \det(V)^{2}$ holds over $B=\Gr\big(5,{\bigwedge}^2 V\big)\setminus Z$.
Since $Z$ has codimension $\ge 1$, it holds over the entire \smash{$\Gr\big(5,{\bigwedge}^2V\big)$}.

Next, since $H^0\big(\Gr\big(5,{\bigwedge}^2V\big),\OO(1)\big)\simeq {\bigwedge}^5\big({\bigwedge}^2V\big)^*$, the map $f$ is given
by a $\GL(V)$-invariant linear map
\begin{gather*}{\bigwedge}^{5}\left({\bigwedge}^2V\right)\to H^0\left(\P V,{\bigwedge}^2T\right)\ot \det(V)^2.\end{gather*}
To show that this map coincides with $\pi_{5,2}$, up to a constant factor, it remains to show
that the space $\Hom_{\GL(V)}\bigl(\bigwedge{5}\big({\bigwedge}^2V\big), H^0\big(\P V,{\bigwedge}^2T\big)\ot \det(V)^2\bigr)$ is $1$-dimensional.

The representation of $\GL(V)$ on $H^0\big(\P V,{\bigwedge}^2T\big)$ is easy to identify due to the exact sequence
\begin{gather*}0\to \k\to V\ot V^*\ot {\bigwedge}^2 V\ot S^2 V^*\to H^0\left(\P V,{\bigwedge}^2T\right)\to 0.\end{gather*}
Using the Littlewood--Richardson rule, we deduce
\begin{gather*}H^0\left(\P V,{\bigwedge}^2T\right)\ot \det(V^*)\simeq \Si^{3,1,1}(V^*),
\end{gather*}
where $\Si^{\la}$ denotes the Schur functor associated with a partition $\la$.
It follows that
\begin{gather*}H^0\left(\P V,{\bigwedge}^2T\right)\ot \det(V)^2\simeq \Si^{3,3,2,2}(V).\end{gather*}
On the other hand, the decomposition of the plethysm $e_5\circ e_2$ (see \cite[Section~I.8, Example~6, p.~138]{Mac}) shows that $\Si^{3,3,2,2}(V)$ appears with multiplicity $1$ in
the $\GL(V)$-representation $\bigwedge^{5}\big({\bigwedge}^2V\big)$. This implies the claimed assertion about $\GL(V)$-maps.
\end{proof}

\subsection[Rank stratification for a bracket of type q\_\{5,2\}]{Rank stratification for a bracket of type $\boldsymbol{q_{5,2}}$}
\label{Rank-strat-sec}

Let $E$ be an elliptic curve, $\VV$ be a stable vector bundle of rank $2$ and degree $5$. We consider the FO bracket $\Pi$ on the projective space
$\P \Ext^1(\VV,\OO)\simeq \P H^0(\VV)^*$.
We want to describe the corresponding rank stratification of $\P H^0(\VV)^*=\P^4$.
More precisely, $\Pi$ is generically nondegenerate, and we are going to determine the degeneration locus $\DD_E\sub \P^4$ (where $\rk \Pi\le 2$) and
the zero locus $S_E$ of $\Pi$.

For every point $p\in E$, we consider the subspace $\La_p:=\VV|_p^*\sub H^0(\VV)^*$ and the corresponding projective line $\P \La_p\sub \P H^0(\VV)^*$.
Recall that the rank of $\Pi$ at a point corresponding to an extension $\wt{\VV}$ is equal to $5-\dim\End\bigl(\wt{V}\bigr)$ (see \cite[Proposition~2.3]{HP-bih}).

\begin{Lemma}\label{Pi-rank-lem}\quad
\begin{itemize}\itemsep=0pt
\item[$(i)$] The bracket $\Pi$ vanishes at the point of $\P \Ext^1(\VV,\OO)$ corresponding to an extension
\begin{gather*}0\to \OO\to \wt{\VV}\to \VV\to 0\end{gather*}
if and only if this extension splits under $\OO\to \OO(p)$ for some point $p\in E$, which happens if and only if
$\wt{\VV}\simeq \OO(p)\oplus \VV'$, where $\VV'$ is semistable of rank $2$ and degree $4$.
Furthermore, in this case $\dim \End(\VV')=2$, so $\VV'$ is either indecomposable, or $\VV'\simeq L_1\oplus L_2$, where $L_1$ and $L_2$ are nonisomorphic line bundles of degree $2$.

\item[$(ii)$] The bracket $\Pi$ has rank $\le 2$ if and only the corresponding extension $\wt{\VV}$ is unstable, or equivalently, there exists a line bundle $L_2$ of degree $2$ such that
the extension splits over the unique embedding $L_2\hra \VV$. In other words, the extension class comes from a subspace of the form
\begin{equation}\label{W-L2-eq}
W_{L_2}:=H^0(L_2)^\perp \sub H^0(\VV)^*=V,
\end{equation}
where we use the unique embedding $L_2\to \VV$ and consider the induced embedding $H^0(L_2)\hra H^0(\VV)$.

\item[$(iii)$] Each plane $\P W_{L_2}\sub \P V$
is a Poisson subvariety, and there is an embedding of the curve~$E$ into $\P W_{L_2}$ by a degree $3$ linear system, so that $\P W_{L_2}\setminus E$ is a symplectic leaf.
\end{itemize}
\end{Lemma}

\begin{proof}
(i) Suppose a nontrivial extension
\begin{gather*}0\to \OO\to \wt{\VV}\to \VV\to 0\end{gather*}
splits under $\OO\to \OO(p)$. Then $\wt{\VV}$ is an extension of $\OO(p)$ by $\VV'$ where $\VV'\sub \VV$ is the kernel of the corresponding
surjective map $\VV\to \OO_p$. Hence, $\VV'$ is semistable of slope $2$, which implies that
\begin{gather*}\wt{\VV}\simeq \OO(p)\oplus \VV'.\end{gather*}
It follows that $\dim \End(\VV')\ge 2$, and so
\begin{gather*}\dim \End\big(\wt{\VV}\big)=3+\dim \End(\VV')\ge 5.\end{gather*}
Hence, $\Pi_E$ vanishes on the points of the line $\P \La_p\sub\P V$, and we have $\dim \End(\VV')=2$, which means that either $\VV'$ is indecomposable
or $\VV'\simeq L_1\oplus L_2$, for two nonisomorphic line bundles $L_1$,~$L_2$ of degree $2$.

Conversely, assume $\Pi$ vanishes at the point corresponding to $\wt{\VV}$, so $\dim \End\big(\wt{\VV}\big)=5$.
Then HN-components of \smash{$\wt{\VV}$} cannot be three line bundles (since they would have to have different positive degrees that add up to $5$),
so $\wt{\VV}=L\oplus \VV'$ where $L$ is a line bundle and $\VV'$ is semistable of rank $2$, $\deg(L)>0$, $0<\deg(\VV')$, $\deg(L)+\deg(\VV')=5$.

The case $\deg(L)=1$ leads to the locus discussed above. If $\deg(L)=2$ and $\deg(\VV')=3$ then $\dim\Hom(\VV',L)=1$, so we get $\dim \End(\VV')=3$ which is impossible.
If $\deg(L)\ge 3$, then~${\deg(\VV')\le 2}$ and $\dim\Hom(\VV',L)\ge 4$, so $\dim \End(\VV)>5$, a contradiction.

(ii) The rank of $\Pi$ is $\le 2$ at $\wt{\VV}$ if and only if $\dim \End\bigl(\wt{\VV}\bigr)\ge 3$. Clearly, such $\wt{\VV}$ has to be unstable. Conversely, any unstable $\wt{\VV}$ would have
form $L\oplus \VV'$ with either $\Hom(L,\VV')\neq 0$ or~${\Hom(\VV',L)\neq 0}$, hence $\dim \End\bigl(\wt{\VV}\bigr)\ge 3$.

Note that \smash{$\mu\big(\wt{\VV}\big)=5/3$}. Hence, if the extension splits over some $L_2\sub \VV$, then $\wt{\VV}$ is unstable. Conversely, if $\wt{\VV}$ is unstable then either it has a line subbundle
of degree $2$, or a semistable subbundle $\VV'$ of rank $2$ and degree $\ge 4$. But any such $\VV'$ has a line subbundle of degree $\ge 2$.

(iii) We can identify $H^0(L_2)^\perp$ with $H^0(L_3)^*\sub H^0(\VV)^*$, where $L_3:=\VV/L_2$. It is easy to see that the intersection of $\P W_{L_2}$ with the zero locus of $\Pi$
is exactly the image of $E$ under the map given by $|L_3|$.

Given an extension $\wt{\VV}\to \VV$, split over $L_2\sub \VV$, the splitting $L_2\to \wt{\VV}$ is unique, and the quotient $\wt{\VV}/L_2$ is an extension of $L_3=\VV/L_2$ by $\OO$.
It is well known that for points of $\P W_{L_2}\setminus E$ the latter extension is stable, so $\VV_{L_3}=\wt{\VV}/L_2$ is a stable bundle of rank $2$ with determinant $L_3$.
Since $\Ext^1(\VV_{L_3},L_2)=0$, we deduce that $\wt{\VV}=\VV_{L_3}\oplus L_2$.
Now we can calculate the image of the map \eqref{conormal-map}. The space $\End\big(\wt{\VV}\big)/\lan\id\ran$ has a basis $\lan\id_{L_2}, e\ran$, where $e$ is a generator of
$\Hom(\VV_{L_3},L_2)$. Their images under \eqref{conormal-map} both factor through $L_2\to E$, hence the image of \eqref{conormal-map} (which is $2$-dimensional)
is $H^0(L_2)\sub H^0(\VV)$. But this is exactly the conormal subspace to the projective plane $\P W_{L_2}$. This shows that $\P W_{L_2}\setminus E$ (and hence
$\P W_{L_2}$) is a Poisson subvariety. Since the rank of $\Pi$ on $\P W_{L_2}\setminus E$ is equal to $2$ and $\Pi|_E=0$, we deduce that $\P W_{L_2}\setminus E$ is a
symplectic leaf.
\end{proof}

By Lemma \ref{Pi-rank-lem}\,(i), the vanishing locus of $\Pi$ corresponds to extensions $\VV$ by $\OO$, which split over $\OO(p)$.
This is the union $S_E$ of the lines $\P \La_p$, where $\La_p=\VV|_p^*\sub \P H^0(\VV)^*$, over $p\in E$.
The surface $S_E$ is the image of the natural map $\P\big(\VV^\vee\big)\to \P(V)$, associated with the embedding of bundles $\VV^\vee\to V\ot \OO_E$.
We will prove (see Lemma \ref{S-lem} below) that in fact this map induces an isomorphism of the projective bundle $\P\big(\VV^\vee\big)$ with $S_E$.

\begin{Lemma}\label{inj-lem}
Let $\EE$ be a vector bundle over a smooth curve $C$ and let $W\to H^0(C,\EE)$ be a linear map from a vector space $W$, such that
for any $x\in C$ the composition $p_x\colon W\to H^0(C,\EE)\to \EE|_x$ is surjective, so that we have a morphism $f\colon\ \P\big(\EE^\vee\big)\to \P(W^*)$. Assume that we have a closed subset $Z\sub \P\big(\EE^\vee\big)$ with the following properties.
\begin{itemize}
\item For every $x,y\in C$, $x\neq y$, consider $p_x(\ker(p_y))\sub \EE|_x$.
Then any $\ell\in \P\big(\EE^\vee|_x\big)$, which is orthogonal to $p_x(\ker(p_y))$, is contained in $Z$.
\item For every $x\in C$, consider the map $W\to H^0(\EE|_{2x})$ and the induced map
\begin{gather*}K_x:=\ker(W\to \EE|_x)\to T_x^*C\ot \EE|_x\end{gather*}
$($where we use the identification $T_x^*C\ot \EE|_x\!=\ker\big(H^0(\EE|_{2x}\big)\!\to \EE|_x))$.
Then any ${\ell\in \P\big(\EE^\vee|_x\big)}$, which is orthogonal to the image of $K_x\ot T_xC$, is contained in $Z$.
\end{itemize}
Then the map $\P\big(\EE^\vee\big)\setminus Z\to \P(W^*)$ is a locally closed embedding.
\end{Lemma}

\begin{proof}
Assume that for $x\neq y$, we have two nonzero functionals $\phi_x\colon\EE|_x\to k$, $\phi_y\colon\EE|_y\to k$ such that $\phi_x\circ p_x=\phi_y\circ p_y$.
Then $(\phi_x\circ p_x)|_{\ker(\phi_y)}=0$. Hence, $\phi_x$ vanishes on $p_x(\ker(p_y))$. By assumption, this can happen only when $\phi_x$ is in $Z$.
Thus, the map from $\P\big(\EE^\vee\big)\setminus Z$ is set-theoretically one-to-one.

Next, we need to check that our map is injective on tangent spaces. The tangent space to~$\P\big(\EE^\vee\big)$ at a point corresponding to $\ell\sub \EE^\vee|_x$ can be
described as follows. Consider the canonical extension
\begin{gather*}0\to T^*_x C\ot \EE|_x\to H^0(\EE|_{2x})\to \EE|_x\to 0.\end{gather*}
Passing to the dual extension of $T_x C\ot \EE^\vee|_x$ by $\EE^\vee|_x$, and restricting it to $T_x C\ot \ell\sub T_x C\ot \EE^\vee|_x$, we get an extension
\begin{gather*}0\to \EE^\vee|_x\to H_\ell\to T_x C\ot \ell\to 0.\end{gather*}
Now the quotient $\bigl(\ell^{-1}\ot H_\ell\bigr)/\k$, where we use the natural embedding
\begin{gather*}k=\ell^{-1}\ot \ell\to \ell^{-1}\ot \EE^\vee|_x\to \ell^{-1}\ot H_{\ell},\end{gather*}
is identified with the tangent space $T_{\ell}\P\big(\EE^\vee\big)$.

The restriction of the map $H^0(\EE|_{2x})^\vee \to W^*$, dual to the natural map $W\to H^0(\EE|_{2x})$, to $H_{\ell}$, induces a map
\begin{gather*}\big(\ell^{-1}\ot H_{\ell}\big)/\k\to W^*/\ell,\end{gather*}
which is exactly the tangent map to $f$. It is injective if and only if the map $H_{\ell}\to W^*$ is injective.
Equivalently, the dual map $W\to H_{\ell}^*$ should be surjective. The latter map is compatible with (surjective) projections to $\EE|_x$,
so this is equivalent to surjectivity of the map
\begin{gather*}K_x=\ker(W\to \EE|_x)\to \ker(H_\ell^*\to \EE|_x)=T_x^*C\ot \ell^{-1}.\end{gather*}
The latter map factors as a composition
\begin{gather*}K_x\to T^*_xC\ot \EE|_x\to T_x^*C\ot \ell^{-1},\end{gather*}
so it is surjective (equivalently, nonzero) if and only if $\ell$ is not orthogonal to the image of $K_x\to T_x^*C\ot \EE|_x$. By assumption, this never
happens for points of $\P\big(\EE^\vee\big)\setminus Z$.
\end{proof}

\begin{Lemma}\label{S-lem}
The map $\P\big(\VV^\vee\big)\to S_E$ is an isomorphism.
\end{Lemma}

\begin{proof} We will check the conditions of Lemma \ref{inj-lem}. It suffices to check surjectivity of the maps $H^0(\VV)\to \VV|_x\oplus \VV|_y$ for $x\neq y$ and
of $H^0(\VV)\to H^0(\VV|_{2x})$. But this follows from the exact sequence
\begin{gather*}0\to \VV(-D)\to \VV\to \VV|_D\to 0\end{gather*}
for any effective divisor $D$ of degree $2$ and from the vanishing of $H^1(\VV(-D))$ by stability of~$\VV$.
\end{proof}

By Lemma \ref{Pi-rank-lem}\,(ii),
the degeneracy locus $\DD_E$ of our Poisson bracket (which is a quintic hypersurface) is the union of planes $\P W_{L_2}\sub \P V$ over $L_2\in \Pic^2(E)$
(see \eqref{W-L2-eq}).
Let us consider the vector bundle $\WW$ over $\wt{E}:=\Pic^2(E)$, such that the fiber of $\WW$ over $L_2$ is $W_{L_2}$.
Note that we have a natural identification $\wt{E}\simeq \Pic^3(E)\colon L_2\mapsto L_3:=\det(\VV)\ot L_2^{-1}$. In terms of $L_3$ we have~${W_{L_2}=H^0(L_3)^*\sub H^0(\VV)^*}$,
where we use a surjection $\VV\to L_3$. To define the vector bundle~$\WW$ precisely, we consider the universal line bundle $\LL_3$ of degree $3$ over
$E\times \wt{E}\simeq E\times \Pic^3(E)$, normalized so that the line bundle $p_{2*}\und{\Hom}(p_1^*\VV,\LL_3)$ is trivial.
We set $\WW:=p_{2*}(\LL_3)^\vee$.
Note that applying $p_{2*}$ to the natural surjection $p_1^*\VV\to \LL_3$ we get a surjection $H^0(\VV)\ot \OO\to p_{2*}(\LL_3)$.
Passing to the dual, we get a morphism $\P(\WW)\to \P V$, whose image is $\DD_E$.

\begin{Lemma} The morphism $\P(\WW)\to \DD_E$ is an isomorphism over $\DD_E\setminus S_E$.
\end{Lemma}

\begin{proof} We need to check two conditions of Lemma \ref{inj-lem} for the morphism $H^0(\VV)\ot \OO\to \WW^\vee$ over $\wt{E}$, with $Z\sub \P(\WW)$
being the preimage of $S_E$. Note that the intersection of $Z$ with each plane $\P H^0(L_3)^*\sub H^0(\VV)^*$ is the elliptic curve $E$ embedded by the
linear system $|L_3|$.

To check the first condition, we use the exact sequence
\begin{gather*}0\to H^0(L_2)\to H^0(\VV)\to H^0(L_3)\to 0,\end{gather*}
where $L_2\ot L_3\simeq \det(\VV)$. If $L'_3$ is different from $L_3$ then
the composed map $L_2\to \VV\to L'_3$ is nonzero, hence, it identifies $L_2$ with the subsheaf $L'_3(-p)$ for some point $p\in E$. Hence,
the image of $H^0(L_2)$ is precisely the plane $H^0(L'_3(-p))\sub H^0(L'_3)$. Hence, the only point of~$\P H^0(L'_3)^*$ orthogonal to this plane
is the point $p\in E\sub \P H^0(L'_3)^*$, which lies in $Z$.

To check the second condition, we need to understand the map $H^0(\VV)\to H^0(\WW^\vee|_{2x})$ for~${x\in \wt{E}\simeq \Pic^3(E)}$. For this
we observe that this map is equal to the composition
\begin{gather*}H^0(\VV)\to H^0(E\times \{2x\},p_1^*\VV|_{E\times \{2x\}})\to H^0(E\times \{2x\},\LL_3|_{E\times \{2x\}}),\end{gather*}
which is the map induced on $H^0$ by the morphism of sheaves on $E$,
\begin{gather*}\a\colon\ \VV\to \VV\ot H^0(\OO_{2x})=p_{1*}(p_1^*\VV|_{E\times \{2x\}})\to p_{1*}(\LL_3|_{E\times \{2x\}}).\end{gather*}

Note that for $x=L_3$, the bundle $F_x:= p_{1*}(\LL_3|_{E\times \{2x\}})$ on $E$ is an extension of $L_3$ by $T_x^*\wt{E}\ot L_3$,
which gives the Kodaira--Spencer map for the family $\LL_3$, so this extension is nontrivial.
The composition
\begin{gather*}\VV\rTo{\a} F_x\to L_3\end{gather*}
is the canonical surjection with the kernel $L_2\sub \VV$.
Hence, $\a$ fits into a morphism of exact sequences
\[
\begin{diagram}
0&\rTo{}& L_2 &\rTo{}& \VV &\rTo{}& L_3 &\rTo{}& 0\\
&&\dTo{\a|_{L_2}}&&\dTo{\a}&&\dTo{\id}\\
0&\rTo{}& T_x^*\wt{E}\ot L_3 &\rTo{}& F_x &\rTo{}& L_3 &\rTo{}& 0.
\end{diagram}
\]
Note that the map $\a|_{L_2}$ is nonzero, since otherwise we would get a splitting of the extension~${F_x\to L_3}$.

Now the kernel of the map $H^0(\VV)\to \WW^\vee|_x=H^0(L_3)$ is identified with $H^0(L_2)$,
and the induced map $H^0(L_2)\to T_x^*\wt{E}\ot H^0(L_3)$ is given by a nonzero map
\begin{gather*}\a|_{L_2}\colon\ L_2\to T_x^*\wt{E}\ot L_3\simeq L_3.\end{gather*}
Hence, its image is the subspace of the form $H^0(L_3(-p))$, and we again deduce that any point of $\P H^0(L_3)^*$ orthogonal to it lies in $Z$.
\end{proof}

\begin{Corollary}\label{lines-cor}\quad
\begin{itemize}\itemsep=0pt
\item[$(i)$] There is a regular map $\DD_E\setminus S_E\to \wt{E}$ such that the fiber over $L_2$ is the symplectic leaf~${\P W_{L_2}\setminus E}$.
\item[$(ii)$] Any line contained in $\DD_E$ is either contained in $S_E$ {\rm (}and so has form $\P \La_p$ for some $p\in E)$ or in some plane $\P W_{L_2}$, where $L_2\in \Pic^2(E)$.
    \end{itemize}
\end{Corollary}

\begin{proof} For (ii) we observe that given a line $L\sub \DD_E$ not contained in $S_E$, the restriction of the map $\DD_E\setminus S\to \wt{E}$
to $L\setminus S_E\to \wt{E}$ is necessarily constant. Hence, $L$ is contained in some plane~$\P W_{L_2}$. Similarly, we have a fibration $S_E\to E$ with fibers $\P \La_p$,
so any line contained in~$S_E$ is one of the fibers.
\end{proof}

\subsection[Two-dimensional distribution on G(2,5) associated with the elliptic curve]{Two-dimensional distribution on $\boldsymbol{G(2,5)}$ associated \\ with the elliptic curve}\label{Td-distr-sec}

Let $E=E_W\sub G(2,V)$ be the elliptic curve obtained as the intersection with a linear subspace~\smash{${\P W\sub \P\big({\bigwedge}^2 V\big)}$} in the Pl\"ucker embedding,
where $\dim W=5$. Equivalently, $E$ is cut out by the linear subspace of sections $W^\perp\sub {\bigwedge}^2 V^*\simeq H^0(G(2,V),\OO(1))$.
As before, we denote by~$\VV$ the restriction of $\UU^\vee$, the dual of the universal bundle. Then ${\bigwedge}^2(\VV)$ is the restriction of $\OO(1)$, and we have
an exact sequence
\begin{gather*}0\to W^\perp \to {\bigwedge}^2 V^*\to H^0\left(E,{\bigwedge}^2(\VV)\right)\to 0.\end{gather*}
In other words, we can identify the dual map
to the embedding $W\hra {\bigwedge}^2 V$ with the natural map
\begin{gather*}{\bigwedge}^2 H^0(\VV)\to H^0\left({\bigwedge}^2\VV\right).\end{gather*}

We have a regular map $f\colon\ G(2,V)\setminus E \to \P^4$
given by the linear system $|W^\perp|\sub |\OO(1)|$.

\begin{Definition}
For every point $\La\in G(2,V)\setminus E$, we define the subspace
\begin{gather*}D_{\La}=D_{E,\La}\sub T_{\La}G(2,V)\end{gather*}
as the kernel of the tangent map to $f$ at $\La$.
\end{Definition}

Note that for generic $\La$, one has $\dim D_{\La}=2$.
We have the following characterization of $D_{\La}$.

\begin{Lemma}\label{Dp-lem}
Let $\La\sub V$ be a
$2$-dimensional subspace corresponding to a point of $G(2,V)\setminus E$.

\begin{itemize}
\item[$(i)$] Under the identification
$T_{\La}G(2,V)\ot \det(\La)\simeq \La\ot V/\La$, we have
\begin{gather*}D_{\La}\ot \det(\La)=W\cap (\La\we V)=W\cap (\La\ot V/\La),\end{gather*}
where the second intersection is taken in ${\bigwedge}^2V/{\bigwedge}^2\La$.

\item[$(ii)$] For each $v\in \La$, let us denote by $\pi_v\colon T_{\La}G(2,V)\to V/\La$ the natural projection.
Assume that $\Pi_{E,v}$ has rank $4$, for some nonzero $v\in \La$. Then $D_{\La}$ is $2$-dimensional, and $\pi_v(D_{\La})$ is the
$2$-dimensional subspace of $V/\La$ given as follows:
\begin{gather*}\pi_v(D_{\La})=\big\{x\in V/\La \mid x\we \Pi_{E,v}^{\rm norm}=0\big\},\end{gather*}
where $\Pi_{E,v}^{\rm norm}\in {\bigwedge}^2(V/\La)$ is the image of $\Pi_{E,v}\in {\bigwedge}^2(V/v)$.
\end{itemize}
\end{Lemma}

\begin{proof}
(i) The map $f$ is the composition of the Pl\"ucker embedding $G(2,V)\to \P\big({\bigwedge}^2 V\big)$ with the linear projection
\begin{gather*}\P \left({\bigwedge}^2 V\right)\setminus \P(W)\to \P\left({\bigwedge}^2 V/W\right).\end{gather*}
Thus, the tangent map to $f$ at $\La\sub W$ is the composition
\begin{gather*}\Hom(\La,V/\La)\rTo{\a} \Hom\left({\bigwedge}^2 \La, {\bigwedge}^2 V/{\bigwedge}^2 \La\right)\to \Hom\left({\bigwedge}^2 \La, {\bigwedge}^2 V/\left({\bigwedge}^2 \La+W\right)\right),\end{gather*}
where $\a(A)(l_1\we l_2)=Al_1\we l_2+l_1\we Al_2 \mod {\bigwedge}^2 \La$. Equivalently, the map $\a$ is the natural map
\begin{gather*}\Hom(\La,V/\La)\simeq \La^*\ot V/\La\simeq {\det}^{-1}(\La)\ot \La\ot V/\La\to {\det}^{-1}(\La)\ot {\bigwedge}^2 V/{\bigwedge}^2 \La,\end{gather*}
given by $l\ot (v\mod \La)\mapsto l\we v \mod {\bigwedge}^2 \La$.

Now the assertion follows from the identification
\begin{gather*}W=\ker\left({\bigwedge}^2 V/{\bigwedge}^2 \La\to {\bigwedge}^2 V/\left({\bigwedge}^2 \La+W\right)\right).\end{gather*}

(ii) Our identification of $\Pi_W$ from Theorem \ref{TheoremA} implies the following property of the bivector~${\Pi_{W,v}\in {\bigwedge}^2(V/v)}$.
Consider the natural map $\phi_v\colon W\to {\bigwedge}^2(V/v)$.
Recall that $S=S_E\sub \P V$ denotes the surface, obtained as the union of lines corresponding to $E\sub G(2,V)$.
We claim that the map~$\phi_v$ is injective if and only if $\lan v\ran$ is not in $S$. Indeed, an element in the kernel of~$\phi_v$ is an element~$v\we v'$ contained in~$W$,
so the plane $\lan v,v'\ran$ corresponds to a point of $E$. Hence, this is true when $\Pi_{W,v}$ is nonzero.

Now assume the rank of $\Pi_{W,v}$ is $4$.
We have a nondegenerate symmetric pairing on ${\bigwedge}^2(V/v)$ with values
in $\det(V/v)$, given by the exterior product. Now our description of $\Pi_W$ implies that for $\lan v\ran\not\in S$, $\Pi_{W,v}$ is nonzero and
\begin{gather*}\phi_v(W)=\lan\Pi_{W,v}\ran^\perp.\end{gather*}

Since $\Pi_{W,v}$ has maximal rank, the skew-symmetric form $(x_1,x_2)=x_1\we x_2\we\Pi_{W,v}$ on $V/v$ is nondegenerate.
Hence, the subspace $(\La/\lan v\ran)\ot (V/\La)$ cannot be contained in $\lan\Pi_{W,v}\ran^\perp$ (this would mean that
$\La/\lan v\ran$ lies in the kernel of $(\cdot,\cdot)$).
Hence, the intersection
\begin{gather*}I:=(\La/\lan v\ran)\ot (V/\La)\cap \lan\Pi_{W,v}\ran^\perp\end{gather*}
is $2$-dimensional.
Since the subspace $\phi_v(W\cap (\La\we V))$ is contained in $I$, we deduce that its dimension is $\le 2$, and so $\dim D_{\La}\le 2$.
But we also know that $\dim D_{\La}\ge 2$, hence in fact, we have $\dim D_{\La}=2$ and $\phi_v(W\cap (\La\we V))=I$.

The last assertion follows from the fact that under trivialization of $\La/\lan v\ran$, the subspace $I\sub V/\La$ coincides with $\pi_v(D_{\La})$.
\end{proof}

\begin{Definition}
We define $\Si_E\sub G(2,V)$ as the closed locus of points $\La\in G(2,V)$ such that~${\dim W\cap (\La\wedge V)\ge 3}$.
\end{Definition}

\begin{Lemma}\label{E-SiE-lem}
One has $\Si_E\sub G(2,V)\setminus E$.
\end{Lemma}

\begin{proof}
We have to prove that $\dim W\cap (\La_p\wedge V)\le 2$,
where $\La_p=H^0(\VV|_p)^*\sub H^0(\VV)^*=V$ for some $p\in E$.
We have, $\La_p^{\perp}=H^0(\VV(-p))\sub H^0(\VV)$ and so, $V/\La_p\simeq H^0(\VV(-p))^*$.

The intersection $W\cap (\La_p\wedge V)$ is the kernel of the composed map
\begin{gather*}W\hra {\bigwedge}^2 V\to {\bigwedge}^2(V/\La_p).\end{gather*}
The dual map can be identified with the composition
\begin{gather*}{\bigwedge}^2 H^0(\VV(-p))\to {\bigwedge}^2 H^0(\VV)\to H^0(\det \VV),\end{gather*}
which also factors as the composition
\begin{gather*}{\bigwedge}^2 H^0(\VV(-p))\to H^0\left({\bigwedge}^2(\VV(-p))\right)=H^0((\det \VV)(-2p))\sub H^0(\det \VV).\end{gather*}
We need to check that this map has corank $2$, or equivalently the first arrow is an isomorphism.

Set $\VV'=\VV(-p)$. This is a stable bundle of rank $2$ and degree $3$. We need to check that the map
\begin{gather*}{\bigwedge}^2 H^0(\VV')\to H^0(\det \VV')\end{gather*}
is surjective. For any point $p'\in E$, we have an exact sequence
\begin{gather*}0\to H^0(\OO(p'))\to H^0(\VV')\to H^0((\det \VV')(-p'))\to 0\end{gather*}
and it is easy to see that the restriction of the above map to $H^0(\OO(p'))\we H^0(\VV')$ surjects onto the subspace
$H^0((\det \VV')(-p'))\sub H^0(\det \VV')$. Varying the point $p'\in E$, we get the needed surjectivity.
\end{proof}

Thus, by Lemma \ref{Dp-lem}\,(i), $\Si_E$ is exactly the set of points $\La\in G(2,V)\setminus E$ where $\dim D_{\La}\ge 3$.
We have the following geometric description of $\Si_E$. Recall that we have a collection of $3$-dimensional subspaces $W_q\sub V$,
associated with points of $\wt{E}=\Pic^2(E)$ (see \eqref{W-L2-eq}).

\begin{Proposition}\label{Si-E-prop}
For $\La\in G(2,V)$, we have $\La\in \Si_E$ if and only if the corresponding line $\P \La$ is contained in some plane $\P W_q$, where $q\in \wt{E}$.
In other words, $\Si_E=\bigcup_{q\in \wt{E}}G(2,W_q)$.
\end{Proposition}

\begin{proof} Assume first that $\La\in \Si_E$. As we have seen above, this means that $\La\in G(2,V)\setminus E$ and $\dim D_{\La}\ge 3$.
By Lemma \ref{Dp-lem}\,(ii), this implies that the rank of the Poisson bracket $\Pi_W$ on points of $\P \La$
is $\le 2$. Hence, by Lemma \ref{Pi-rank-lem}\,(ii), $\P \La$ is contained in the quintic $\DD_E$. By Corollary~\ref{lines-cor}, this implies that $\P \La$ is contained in some plane $\P W_q$.

Conversely, assume that we have a $2$-dimensional subspace $\La\sub H^0(M)^*\sub H^0(\VV)^*=V$, where $\VV\to M$ is a surjection to a degree $3$ line bundle $M$.
Then $\La=\lan s\ran^\perp\sub H^0(M)^*$ for some $1$-dimensional subspace $\lan s\ran\sub H^0(M)$.
Set $P=\La^\perp\sub H^0(\VV)$. Then $P$ is the preimage of~${\lan s\ran\sub H^0(M)}$ under the projection $H^0(\VV)\to H^0(M)$.

By Lemma \ref{Dp-lem}, the space $D_{\La}$ is isomorphic to the kernel of the composed map
\begin{gather*}W\to {\bigwedge}^2 V\to {\bigwedge}^2(V/\La).\end{gather*}
Hence, $\dim(D_{\La})$ is equal to the corank of the dual map
\begin{equation}\label{wedge2-P-map}
{\bigwedge}^2(P)\to {\bigwedge}^2 H^0(\VV)\to H^0\left({\bigwedge}^2\VV\right).
\end{equation}

Let $B$ denote the divisor of zeroes of $s$. We claim that the image of \eqref{wedge2-P-map} is contained in the subspace
$H^0\big({\bigwedge}^2\VV(-B)\big)\sub H^0\big({\bigwedge}^2\VV\big)$. Indeed, we have an exact sequence
\begin{gather*}0\to N\to \VV\to M\to 0,\end{gather*}
where $N$ is a line bundle of degree $2$. It is easy to see that the composed map
\begin{gather*}H^0(N)\we H^0(\VV)\hra {\bigwedge}^2 H^0(\VV)\to H^0\left({\bigwedge}^2\VV\right)\end{gather*}
coincides with the natural multiplication map
\begin{gather*}H^0(N)\we H^0(\VV)/{\bigwedge}^2H^0(N)\simeq H^0(N)\ot H^0(M)\to H^0(N\ot M)\simeq H^0\left({\bigwedge}^2\VV\right).\end{gather*}
The exact sequence
\begin{gather*}0\to H^0(N)\to P\to \lan s \ran\to 0\end{gather*}
shows that ${\bigwedge}^2 P\sub H^0(N)\we H^0(\VV)$ and its image in $H^0(N)\ot H^0(M)$ is contained in $H^0(N)\ot \lan s\ran$.
This proves our claim about the image of the map \eqref{wedge2-P-map}.
It follows that the corank of this map is $\ge 3$, so $\La\in \Si_E$.
\end{proof}

\begin{Corollary}\label{lines-Si-cor}
The locus of lines in $\P^4$ contained in the degeneration locus $\DD_E$ of $\Pi_E$ corresponds to the union $E\sqcup \Si_E\sub G(2,V)$.
\end{Corollary}

\begin{proof} Combine Proposition \ref{Si-E-prop} with Corollary \ref{lines-cor}\,(ii). The union is disjoint by Lem\-ma~\ref{E-SiE-lem}.
\end{proof}

\begin{Lemma}\label{Dp-subbund-lem}
Let $\La\in G(2,V)\setminus E$.
\begin{itemize}\itemsep=0pt
\item[$(i)$] For any $3$-dimensional subspace $M\sub V$ containing $\La$, one has $W\cap {\bigwedge}^2 M={\bigwedge}^2 \La$.
\item[$(ii)$] Assume that for generic $v\in \La$, the rank of $\Pi_{E,v}$ is $4$.
Then the map $D_{\La}\ot \OO\to V/\La\ot \OO(1)$ over the projective line $\P \La$ is an embedding of a rank $2$ subbundle.
\end{itemize}
\end{Lemma}

\begin{proof}
(i) Since all elements of ${\bigwedge}^2 M$ are decomposable, the intersection $Q:=W\cap {\bigwedge}^2 M$ is a~linear subspace consisting of
decomposable elements. But all decomposable elements of $W$ are of the form ${\bigwedge}^2 \La_p$ for some point $p\in E$. Hence, we would get
an embedding $\P(Q)\to E$, which implies that $Q$ is $1$-dimensional, so $Q={\bigwedge}^2 \La$.

(ii) From part (i) and from Lemma \ref{Dp-lem} we get that for any $3$-dimensional subspace $M\sub V$ containing $\La$, one has $D_{\La}\cap \La\ot M/\La=0$.
Let us set $P=V/\La$, and let us consider the exact sequence over $\P \La$,
\begin{gather*}0\to D_{\La}\ot \OO(-1)\to P\ot \OO\to Q\to 0.\end{gather*}
We want to prove that the rank $1$ sheaf $Q$ on $\P^1$ has no torsion. Since $\deg(Q)=2$ and~$Q$ is generated by global sections, we only have to exclude the possibilities
$Q\simeq \OO_x\oplus \OO(1)$ and~${Q\simeq T\oplus \OO}$, where $T$ is a torsion sheaf of length $2$.

Assume first that $Q\simeq \OO_x\oplus \OO(1)$. Consider the composed surjection $f\colon P\ot \OO\to Q\to \OO(1)$. It is induced by a surjection $P\to H^0(\OO(1))$,
which has $1$-dimensional kernel $\lan v\ran$. It follows that the inclusion of $D_{\La}\ot \OO(-1)$ into $P\ot \OO$ factors as
\begin{gather*}D_{\La}\ot \OO(-1)\to \lan v\ran\ot \OO\oplus \OO(-1)\to P\ot \OO.\end{gather*}
Hence, $D_{\La}$ has a nontrivial intersection with $H^0(\OO(1))\ot \lan v\ran=\La\ot M/\La\sub \La\ot V/\La$, for some $3$-dimensional $M\sub V$, containing $\La$.
This is a contradiction, as we proved that there could be no such $M$.

In the case $Q\simeq T\oplus \OO$, we get that $D_{\La}\ot \OO(-1)$ is contained in the kernel of a surjection~${P\ot \OO\to \OO}$, i.e., $D_{\La}\ot \OO(-1)$ is contained in
$\OO^2\sub P\ot \OO$. But any embedding $\OO(-1)^2\to \OO^2$ factors through some $\OO(-1)\oplus \OO\to \OO^2$ (occurring as kernel of the surjection $\OO^2\to \OO_x$,
for some point $x$ in the support of the quotient). Hence, we can finish again as in the previous case.
\end{proof}

\begin{Remark} The rational map $f$ from $G(2,V)$ to $\P^4$ has the following interpretation, which can be proved using projective duality.
Start with a generic line $\ell\sub \P(V)$. Then the intersection~${\ell\cap \DD_E}$ with the degeneration quintic of $\Pi_E$ consists of $5$ points.
Taking the images of these points under the projection $\DD_E\setminus S_E\to \wt{E}$ (see Corollary~\ref{lines-cor}) we get a divisor $D_{\ell}$ of degree~$5$ on~$\wt{E}$.
All these divisors will belong to a certain linear system $\P^4$ of degree $5$, and the map~${\ell\mapsto D_{\ell}}$ is exactly our map $f$.
\end{Remark}

\subsection[Calculation of the Schouten bracket and proof of Theorem B]{Calculation of the Schouten bracket and proof of Theorem \ref{TheoremB}}

\begin{Lemma}\label{conormal-Lie-lemma}\quad
\begin{itemize}\itemsep=0pt
\item[$(i)$] Let $E\sub G(2,V)$ be the elliptic curve defined by $W\sub {\bigwedge}^2 V$. Then for each point $p\in E$, the bivector $\Pi_E$ vanishes on the projective line
$\P \La_p\sub \P V$, where $\La_p\sub V$ is the $2$-dimensional subspace corresponding to $p$. For a generic point $v$ of $\La_p$, the Lie algebra $\fg=T^*_v \P V$
has a~basis $(h_1,h_2,e_1,e_2)$ such that
\begin{gather*}[h_1,h_2]=[e_1,e_2]=0,\\
[h_i,e_i]=2e_i, \qquad [h_j,e_i]=-e_i \qquad \text{for}\quad i\neq j.\end{gather*}
Equivalently,
the linearization of $\Pi_E$ takes form
\begin{gather*}\Pi^{\rm lin}_E=2e_1\pa_{h_1}\we\pa_{e_1}-e_1\pa_{h_2}\we\pa_{e_1}+2e_2\pa_{h_2}\we\pa_{e_2}-e_2\pa_{h_1}\we\pa_{e_2}.\end{gather*}
Furthermore, the conormal subspace $N_{\P \La_p, v}^\vee \sub \fg^*$ is spanned by $e_1$, $e_2$, $h_1+h_2$ {\rm(}dually the tangent space to $T_{\P \La_p}$ is spanned by
$\pa_{h_1}-\pa_{h_2})$.

\item[$(ii)$] We have an identification
\begin{gather*}H^0(\P \La_p, N_{\P \La_p})\simeq H^0(\P \La_p, V/\La_p\ot \OO(1))\simeq \La_p^*\ot V/\La_p\simeq T_p G(2,V).\end{gather*}
Under this identification, the line $T_p E\sub T_p G(2,V)$ has the property that
the corresponding global section of $N_{\P \La_p}$ evaluated at generic $v\in \P \La_p$ spans the line
\begin{gather*}\lan \pa_{h_1},\pa_{h_2}\ran/\lan \pa_{h_1}-\pa_{h_2}\ran\sub N_{\P \La_p, v}\simeq V/\La_p.\end{gather*}
Equivalently, the tangent space at $v$ to the surface $S_E\sub \P V$ is $\lan \pa_{h_1},\pa_{h_2}\ran\sub T_v \P V$.

\item[$(iii)$] Let $\Pi'$ be a Poisson bracket compatible with $\Pi_E$. Then for $p\in E$ and a generic $v\in \La_p$, one has
\begin{equation}
\label{comp-bracket-form}
\Pi'_v\in \lan (2\pa_{h_1}-\pa_{h_2})\we \pa_{e_1}, (2\pa_{h_2}-\pa_{h_1})\we \pa_{e_2}, \pa_{h_1}\we \pa_{h_2}\ran.
\end{equation}
\end{itemize}
\end{Lemma}

\begin{proof}
(i) Extensions $\wt{\VV}$ of $\VV$ by $\OO$, corresponding to the line $\P \La_p$, are exactly nontrivial extensions that split under $\OO\to \OO(p)$.
We claim that for a generic point of $\P \La_p$ we have \smash{$\wt{\VV}\simeq \OO(p)\oplus L_1\oplus L_2$}, where $L_1$ and $L_2$ are nonisomorphic line bundles
of degree $2$. Indeed, by Lemma \ref{Pi-rank-lem}\,(ii), the only other possibility is $\wt{\VV}\simeq \OO(p)\oplus \VV'$, where $\VV'$ is a nontrivial extension of
$M$ by $M$, where $M^2\simeq \det(\VV)(-p)$. Since the corresponding extension splits over the unique embedding $M\to \VV$, this gives one point on the line $\P \La_p$
for each of the four possible line bundles~$M$.

We can compute the Lie algebra $\fg$ for the point corresponding to $\wt{\VV}\simeq \OO(p)\oplus L_1\oplus L_2$ using the isomorphism of
Theorem \ref{conormal-thm},
\begin{equation}\label{End-to-conormal-map}
\End\big(\wt{\VV}\big)/\lan \id\ran\rTo{\sim} \fg\sub H^0(\VV).
\end{equation}
We consider
the following basis in $\End\big(\wt{\VV}\big)/\lan \id\ran$:
\begin{gather*}h_i=\id_{L_i}-\id_{\OO(p)}, \qquad e_i \in \Hom(\OO(p),L_i), \quad i=1,2.\end{gather*}
Then it is easy to check the claimed commutator relations between these elements.

The conormal subspace to $\P \La_p$ is identified with $\La_p^\perp=H^0(\VV(-p))$. The image of the subspace~${\Hom(\OO(p),L_1\oplus L_2)}$ under the map
\eqref{End-to-conormal-map} will consist of compositions
\begin{gather*}\OO\to \OO(p)\to L_1\oplus L_2\to \VV,\end{gather*}
which vanish at $p$, so they are contained in $H^0(\VV(-p))$.
We have
\begin{gather*}h_1+h_2=\id_{L_1}\oplus \id_{L_2}-2\id_{\OO(p)}\equiv -3\id_{\OO(p)} \mod \lan\id_{\wt{\VV}}\ran,\end{gather*}
and the element $\id_{\OO(p)}$ is mapped under \eqref{End-to-conormal-map} to the composition
\begin{gather*}\OO\to \OO(p)\to \VV,\end{gather*}
which also vanishes at $p$.
This proves our claim about the conormal subspace.

(ii) To identify the direction corresponding to $T_pE$, we first recall that the map $E\to G(2,V)$ is associated with
the subbundle $\VV^\vee\hra V\ot \OO$ over $E$. We have an exact sequence
\begin{gather*}0\to T_p^*E\ot \VV|_p\to H^0(\VV|_{2p})\to \VV|_p\to 0.\end{gather*}
The dual of the natural map $V^*\to H^0(\VV|_{2p})$ fits into a morphism of exact sequences
\[
\begin{diagram}
0&\rTo{}&\VV^\vee|_p&\rTo{}&H^0(\VV|_{2p})^*&\rTo{}& T_pE\ot \VV^{\vee}|_p&\rTo{}&0\\
&&\dTo{\sim}&&\dTo{}&&\dTo{\b}\\
0&\rTo{}&\La_p&\rTo{}&V&\rTo{}& V/\La_p&\rTo{}&0
\end{diagram}
\]
and the map $\b$ corresponds to a map $T_pE\to \Hom(\VV^\vee|_p,V/\La_p)=\Hom(\La_p,V/\La_p)$
which is the tangent map to $E\to G(2,V)$.
Note that the dual to $\b$ is the natural linear map
\begin{equation}\label{TpE-linear-map-eq}
(V/\La_p)^*=\ker\big(H^0(\VV)\to \VV|_p\big)\to \ker\big(H^0(\VV|_{2p})\to \VV|_p\big)\simeq T_p^*E\ot \VV|_p.
\end{equation}

Now, given a functional $v\colon \VV|_p\to k$, the image of $T_pE$ under $\pi_v\colon\La_p^*\ot V/\La_p\to V/\La_p$ corresponds
to the composition of \eqref{TpE-linear-map-eq} with $v$. In other words, it is given by the composition
\begin{gather*}\La_p^\perp=H^0(\VV(-p))\to \VV(-p)|_v\simeq \VV|_p\rTo{v} k\end{gather*}
(here we use a trivialization of $T_pE$).

Let $\wt{\VV}\to \VV$ be the extension corresponding to $v$. As we have seen in (i), for a generic~$v$, we have $\wt{V}\simeq \OO(p)\oplus L_1\oplus L_2$,
where $L_i$ are as above. As we have seen in (i), under the isomorphism~\eqref{End-to-conormal-map},
$\La_p^\perp=H^0(\VV(-p))$ is the image of the subspace $\lan h_1+h_2,e_1,e_2\ran$.

Hence, it remains to check that the composition
\begin{gather*}\lan e_1,e_2\ran\to H^0(\VV(-p))\to \VV(-p)|_p\simeq \VV|_p\rTo{v} k,\end{gather*}
is zero (where the first arrow is induced by \eqref{End-to-conormal-map}).
Let us consider the element $e_1$ (the case of~$e_2$ is similar). It maps to the element of $H^0(\VV(-p))$ given by
the embedding
\begin{gather*}\OO\to L_1(-p)\to \VV(-p),\end{gather*}
where we use the composed map $L_1\to \wt{\VV}\to \VV$.
Thus, it is enough to check that the composition \smash{$L_1\to \VV\rTo{v} \OO_p$} is zero.
To this end we use the fact that the extension $\wt{\VV}$ is the pull-back of the standard extension $\OO(p)\to \OO_p$ via $v$. Hence, we have
a commutative diagram with exact rows and columns,
\[
\begin{diagram}
0&\rTo{}&\OO&\rTo{}&\OO(p)&\rTo{}&\OO_p&\rTo{}&0\\
&&\uTo{\id}&&\uTo{}&&\uTo{v}\\
0&\rTo{}&\OO&\rTo{}&\wt{\VV}&\rTo{}&\VV&\rTo{}&0\\
&&&&\uTo{}&&\uTo{}\\
&&&&L_1\oplus L_2&\rTo{\id}& L_1\oplus L_2
\end{diagram}
\]
which shows that the composition $L_1\oplus L_2\to \VV\to \OO_p$ is zero.

(iii)
This is obtained by a straightforward computation using the vanishing of $[\Pi_E,\Pi']$ and the formula for $\Pi_E^{\rm lin}$ from part (i).
\end{proof}

\begin{Lemma}\label{comp-Si-lem}
Let $E,E'\sub G(2,V)$ be a pair of elliptic curves obtained as linear sections,
such that $[\Pi_E,\Pi_{E'}]=0$. Then
$E$ is not contained in $\Si_{E'}\sub G(2,V)$.
\end{Lemma}

\begin{proof}
Assume $E\sub \Si_{E'}$. Then, by the description of $\Si_{E'}$ in Proposition \ref{Si-E-prop}, for every $p\in E$ there exists a line bundle $L_2$ of degree $2$ on $E'$ such that
the image of $H^0(\VV|_p)^*\to H^0(E,\VV)^*=V$ is contained in $H^0(E',L_2)^\perp\sub H^0(E',\VV')^*=V$.
In other words, each line $\P \La_p\sub \P V$, for $p\in E$, is contained in
the projective plane $\P H^0(E',L_2)^\perp \sub \P V$.
This plane intersects the zero locus of $\Pi_{E'}$ in a smooth cubic (see Lemma \ref{Pi-rank-lem}\,(iii)), hence, for a generic point $v\in \La_p$, the rank of~$\Pi_{E'}|_v$ is $2$.

Hence, $\Pi_{E'}|_v=w_1\we w_2$, where $\lan w_1,w_2\ran$ is the tangent plane to the leaf of $\Pi_{E'}$ (i.e., to the projective plane $\P H^0(E',L_2)^\perp$).
Furthermore, the plane $\lan w_1,w_2\ran$ contains the tangent line to~$\P \La_p$ at $v$. In the notation of Lemma \ref{conormal-Lie-lemma}\,(i),
the latter tangent line is spanned by $\pa_{h_1}-\pa_{h_2}$.
So, $\Pi_{E'}|_v=(\pa_{h_1}-\pa_{h_2})\we w$ for some tangent vector $w$. But we also know by Lemma \ref{conormal-Lie-lemma}\,(iii) that~$\Pi_{E'}|_v$ is a linear combination of
$(2\pa_{h_1}-\pa_{h_2})\we \pa_{e_1}$, $(2\pa_{h_2}-\pa_{h_1})\we \pa_{e_2}$ and $\pa_{h_1}\we \pa_{h_2}$.
This is possible only when $w\in \lan \pa_{h_1}, \pa_{h_2}\ran$, which is the tangent plane to the surface $S_E$ (see Lemma~\ref{conormal-Lie-lemma}\,(ii)).

This implies that $S_E$ is tangent to the corresponding projective plane $\P H^0(E',L_2)^\perp\sub \DD_{E'}$.
Assume first that $S_E\not\sub S_{E'}$. Then we get that the regular morphism
\begin{gather*}S_E\setminus S_{E'}\to \DD_{E'}\setminus S_{E'}\to \Pic^2(E')\end{gather*}
(see Corollary \ref{lines-cor}) has zero tangent map at every point. Hence, $S_E$ is contained in a projective plane, which is a contradiction
(since the map $\P\big(\VV^\vee\big)\to \P H^0(\VV)^*=\P V$ induces an isomorphism on sections of $\OO(1)$).

Finally, if $S_E\sub S_{E'}$ then $E=E'\sub G(2,V)$ and, we get a contradiction by Lemma \ref{E-SiE-lem}.
\end{proof}

\begin{proof}[Proof of Theorem \ref{TheoremB}]
(i) We can assume that $E\neq E'$. We will check that for a generic point~${p\in E}$, one has
\begin{equation}\label{T-D-inclusion}
T_pE \sub D_{E',p}\sub T_p G(2,V).
\end{equation}

By Lemma \ref{comp-Si-lem},
for a generic $p\in E$, we have $p\not\in \Si_{E'}$, hence, by Corollary \ref{lines-Si-cor}, the line~$\P \La_p$ is not contained in the degeneracy locus $\DD_{E'}$ of $\Pi_{E'}$.
Let us pick a generic point $v$ of $\La_p$, so that the rank of $\Pi_{E',v}$ is $4$.
We want to study the normal projection
\begin{gather*}\Pi_{E',v}^{\rm norm}\in \wedge^2(T_v\P V/T_v \P \La_p)\simeq \wedge^2(V/\La_p)\end{gather*}
(see Lemma \ref{Dp-lem}).

Recall that in the notation of Lemma \ref{conormal-Lie-lemma}, the tangent space to $\P \La_p$ at $v$ is spanned by~${\pa_{h_1}-\pa_{h_2}}$.
Hence, the inclusion \eqref{comp-bracket-form} implies that $\Pi_{E',v}^{\rm norm}$
is proportional to a bivector of the form $\pa_{h_1}\we \xi$ .
By Lemma \ref{conormal-Lie-lemma}\,(ii), we can reformulate this as
\begin{gather*}\Pi_{E',v}^{\rm norm}\in \pi_v(T_pE)\we V/\La_p\sub \wedge^2(V/\La_p).\end{gather*}

By Lemma \ref{Dp-lem}\,(ii), the subspace $\pi_v(D_{E',p})\sub V/\La_p$ consists of $x$ such that $x\we \Pi_{E',v}^{\rm norm}=0$.
Thus, we deduce the inclusion
\begin{gather*}\pi_v(T_pE)\sub \pi_v(D_{E',p})\sub V/\La_p\end{gather*}
for generic $v\in \La_p$.

In other words, the section $s$ generating
\begin{gather*}T_pE\sub T_{\La_p}G(2,V)\simeq\Hom(\La_p,V/\La_p)\simeq H^0(\P \La_p,V/\La_p\ot \OO(1))\end{gather*}
has the property that for generic point $v\in \P \La_p$ the evaluation $s(v)$ belongs to the image of the evaluation at $v$ of
the embedding $D_{E',p}\ot \OO\to V/\La_p\ot \OO(1)$. Since by Lemma \ref{Dp-subbund-lem} the latter is an embedding of a subbundle,
this implies that in fact $s\in D_{E',p}$ as claimed.

This proves the inclusion \eqref{T-D-inclusion} for a generic $p\in E$. But this implies that the composed map
\begin{gather*}E\setminus E'\to G(2,V)\setminus E'\to \P^4\end{gather*}
has zero derivative everywhere, so it is constant. Hence, $E$ is contained in a linear section of~${\P U\cap G(2,V)}$, for some $6$-dimensional subspace
$U\sub {\bigwedge}^2V$ containing $W'$. Hence, ${\dim(W+W')\le 6}$.

Conversely, assume $W$ and $W'$ are such that $U=W+W'$ is $6$-dimensional. Then we claim that $[\Pi_W,\Pi_{W'}]=0$.
Indeed, since the space of such pairs $(W, W')$ is irreducible, it is enough to consider the case when the surface $S=\P U\cap G(2,V)$ is smooth.
Then $E_W$ and $E_{W'}$ are anticanonical divisors on $S$, and we can apply \cite[Theorem~4.4]{HP-bih} to the bundle $\VV_S:=\UU^\vee|_S$ on $S$.
The fact that $(\OO_S,\VV_S)$ is an exceptional pair is easily checked using Koszul resolutions, as in Section~\ref{Koszul-sec}.

(ii) It is well known that if a collection of $k$-dimensional subspaces in a vector space has the property that any two subspaces intersect in a $(k-1)$-dimensional space,
then either all of them are contained in a fixed $(k+1)$-dimensional subspace, or they contain a fixed $(k-1)$-dimensional subspace.
The statement immediately follows from (i) using this fact for $k=5$ and the collection~$(W_i)$.
\end{proof}

\begin{proof}[Proof of Corollary~\ref{CorollaryC}]
By Theorem \ref{TheoremB}\,(ii), the brackets $(\Pi_{W_i})$ are pairwise compatible when either there exists a $6$-dimensional subspace $U\sub {\bigwedge}^2 V$, containing all $W_i$,
or there is a $4$-dimensional subspace $K\sub {\bigwedge}^2 V$, contained in all $W_i$. In the former case the corresponding tensors ${\bigwedge}^5 W_i$ are all
contained in the $6$-dimensional subspace
\begin{gather*}{\bigwedge}^5 U\sub {\bigwedge}^5 \left({\bigwedge}^2 V\right).\end{gather*}
In the latter case all the tensors ${\bigwedge}^5 W_i$ are contained in the $6$-dimensional subspace
\begin{gather*}{\bigwedge}^4 K\ot \left({\bigwedge}^2V/K\right)\simeq \left({\bigwedge}^4 K\right)\we \left({\bigwedge}^2 V\right)\sub{\bigwedge}^5 \left({\bigwedge}^2 V\right).\end{gather*}

Conversely, by \cite[Theorem~4.4]{HP-bih}, if we take a smooth linear section $S=\P U\cap G(2,V)$, where~${\dim U=6}$, we claim that we
will get a $6$-dimensional subspace of compatible Poisson brackets coming from anticanonical divisors of $S$.
We just need to show that the corresponding linear map from $H^0\big(S,\om_S^{-1}\big)$ to the space of Poisson bivectors on $\P(V)$ is injective.
Suppose there exists an anticanonical divisor $E_0\sub E$ such that the corresponding Poisson bivector is zero. Pick a generic anticanonical divisor $E$.
Then all elliptic curves in the pencil $E+tE_0$ map to the same Poisson bivector. But this is impossible since we can recover $E\sub G(2,V)$ from
the corresponding Poisson bracket $\Pi_E$ on $\P(V)$, as the set of all lines lying in the zero locus $S_E$ (see Section~\ref{Rank-strat-sec}).
\end{proof}

\subsection*{Acknowledgments} We are grateful to Volodya Rubtsov for useful discussions and to the anonymous referee for helpful comments.
N.M.~would like to thank the Max Planck Institute for Mathematics for
hospitality and perfect work conditions.

\pdfbookmark[1]{References}{ref}
\LastPageEnding

\end{document}